\documentclass{amsart}
\usepackage{amsmath, amsthm,
amssymb,amscd,amsxtra, MnSymbol, pb-diagram, tikz-cd}

\DeclareMathOperator{\Imm}{Im}

\newcommand{\field}[1]{\mathbb{#1}}
\newcommand{\Z}{\field{Z}}
\newcommand{\e}{\epsilon}
\newcommand{\F}{\field{F}}

\newcommand{\N}{\field{N}}

\newcommand{\lr}{{ \longrightarrow }}

\newtheorem {theorem}{Theorem}[section]
\newtheorem {proposition}[theorem]{Proposition}

\newtheorem {lemma}[theorem]{Lemma}
\newtheorem {corollary}[theorem]{Corollary}
\theoremstyle{definition}

\theoremstyle{remark}

\numberwithin{equation}{section}
\begin{document}
\bibliographystyle{plain}
\title[Fractal Structures]{Searching for Fractal Structures in the Universal Steenrod Algebra at Odd Primes}
\author{Maurizio Brunetti}\author{Adriana Ciampella}
\address{\, \newline
Dipartimento di Matematica e applicazioni,\newline
Universit\`a di Napoli \char'134 Federico II'',\newline
 Piazzale Tecchio 80   I-80125 Napoli, 
Italy. \newline \newline
 E-mail: mbrunett@unina.it, ciampell@unina.it}


\begin{abstract}  Unlike the $p=2$ case, the universal Steenrod Algebra ${\mathcal Q} (p) $ at odd primes does not have a fractal structure that preserves the length of monomials. Nevertheless, when $p$ is odd we detect inside $\mathcal Q(p)$  two different families of nested subalgebras each isomorphic (as length-graded algebras) to the respective starting element of the sequence. 
\end{abstract}

\maketitle

\section{Introduction} 
Let $p$ be any prime. The so-called universal Steenrod algebra  ${\mathcal Q}(p)$ is an $\F_p$-algebra extensively studied by the authors (see, for instance, \cite{ManuMath}-\cite{Funda}). On its first appearance, it has been described as the algebra of cohomology operations in the category of $H_\infty$-ring spectra (see \cite{May}). Invariant-theoretic descriptions of ${\mathcal Q(p)}$ can be found in \cite{CL} and \cite{L0}.  
When $p$ is an odd prime, the augmentation ideal of $\mathcal Q(p)$ is the free $\F_p$-algebra over the set 
\begin{equation} \mathcal S_p = \{ \, z_{\epsilon,k} \; \vert \; (\epsilon,k)  \in  \{0,1\} \times \Z \, 	\}
\end{equation}  subject to the set of relations
\begin{equation}\label{relaz} \mathcal R_p= \{ \, R(\epsilon, k,n), \, S(\epsilon, k,n) \; \vert \; (\epsilon, k,n) \in  \{0,1\} \times \Z 	\times \N_0 \, \},
\end{equation}
where
\begin{equation}\label{Rkn}
R(\e, k,n)=z_{\e,pk-1-n}z_{0,k}+\sum_j (-1)^j \binom{(p-1)(n-j)-1}{ j}z_{\e,pk-1-j}z_{0,k-n+j},
\end{equation}
and
\begin{equation}\label{Skn}
\begin{aligned}
S(\e, k,n)&=z_{\e,pk-n}z_{1,k}+\sum_j (-1)^{j+1} \binom{(p-1)(n-j)-1}{ j}z_{\e,pk-j}z_{1,k-n+j}\\
\,&+(1-\e)\sum_j(-1)^{j+1} \binom{(p-1)(n-j)}{ j}z_{1,pk-j}z_{0,k-n+j}.\\
\end{aligned}
\end{equation}
Such relations are known as {\it generalized Adem relations}. 

 The algebra $\mathcal Q(p)$ is related to many Steenrod-like operations. For instance to those acting on the cohomology of a graded cocommutative Hopf algebra  (\cite{Viet}, \cite{Li}), or the Dyer-Lashof operations on the homology of infinite loop spaces (\cite{AK} and \cite{May2}).  Details of such connections, at least for $p=2$, can be found in \cite{BLMS}. In particular, the ordinary Steenrod algebra $\mathcal A(p)$ is a quotient of $\mathcal Q(p)$. At odd primes, the algebra epimorphism  is determined by
\begin{equation}
\zeta :  z_{\e,k} \longmapsto \begin{cases} \beta^{\e} P^k \quad \text{if $k \geq 0$,}\\
0 \quad \qquad \text{otherwise.}
\end{cases}
\end{equation}
The kernel of the map $\zeta$ turns out to be the principal ideal  generated by  $z_{0,0}-1$.

All monic monomials in ${\mathcal Q}(p) $, with the exception of $z_{\emptyset} =1$ have the form 
 \begin{equation}\label{mon} z_I=z_{\e_1,i_1}z_{\e_2,i_2}\cdots z_{\e_m,i_m},
 \end{equation}
  where the string $I=(\e_1,i_1;\e_2, i_2; \dots ;\e_m, i_m)$ is  the {\it label} of the monomial $z_I$. 
  By {\em length}  of a monomial $z_I$ of type \eqref{mon} we just mean the integer $m$, while the length of  any $\rho \in \F_p \subset \mathcal Q(p)$ is defined to be $0$. 
  Since Relations \eqref{Rkn} and \eqref{Skn} are homogeneous with respect to length, the algebra $\mathcal Q(p)$ can be regarded as a graded object.
  
  A monomial and its label are said to  be admissible if $i_j \geq pi_{j+1}+\e_{j+1}$ for any $j=1,\dots , m-1$. We also consider $z_{\emptyset}=1 \in \F_p \subset {\mathcal Q}(p)$ admissible. The set ${\mathcal B}$ of all monic admissible monomials forms an $\F_p$-linear basis for  ${\mathcal Q}(p)$
 (see \cite{CL}).
 
Through two different approaches, in \cite{ApplMS}  and \cite{RendASFM} it has been shown that ${\mathcal Q}(2)$  has a fractal structure given by a sequence of nested subalgras ${\mathcal Q}_s$, each isomorphic to ${\mathcal Q}$. 
The interest in searching for fractal structures inside algebras of (co-)homology operations initially arouse in \cite{Monks}, where such structures were used as a tool to establish the nilpotence height of some elements in $\mathcal A(p)$. Results in the same vein are in \cite{Karaka}. 

Recently, in \cite{BolMex} the authors proved that  no length-preserving strict monomorphisms turn out to exist in $\mathcal Q(p)$ when $p$ is odd. Hence no descending chain of isomorphic subalgebras starting with $\mathcal Q(p)$ exists for $p>2$. Results in \cite{BolMex}  did not exclude the existence of fractal structures for proper subalgebras of $\mathcal Q(p)$. As a matter of fact,  the subalgebras ${\mathcal Q}^0$ and ${\mathcal Q}^1$  generated by the $z_{0,h}$'s and the $z_{1,k}$'s respectively (together with $1$) turn out to have self-similar shapes, as stated in our Theorem \ref{teo1}, our main result.

\begin{theorem}\label{teo1}  Let $p$ be any odd prime. For any $\e \in \{0,1\}$ there is a chain of nested subalgebras of $\mathcal Q(p)$
$$ {\mathcal Q}_0^\e \supset {\mathcal Q}_1^\e \supset {\mathcal Q}_2^\e \supset \dots \supset {\mathcal Q}_s^e \supset {\mathcal Q}_{s+1}^\e\supset \dots $$
each isomomorphic to ${\mathcal Q}_0^\e={\mathcal Q}^\e$ as length-graded algebras. 
\end{theorem}
Theorem \ref{teo1} relies on the existence of two suitable algebra monomorphisms
\begin{equation} \phi : {\mathcal Q}^0 \lr {\mathcal Q}^0 \quad \text{and} \quad  \psi : {\mathcal Q}^1 \lr {\mathcal Q}^1.
\end{equation}
Indeed, we just set  ${\mathcal Q}_s^0=\phi^s({\mathcal Q}^0)$ and ${\mathcal Q}_s^1=\phi^s({\mathcal Q}^1)$, the restrictions $\phi\hspace{-0.15em}\mid_{{\mathcal Q}_s^0}$ and  $\psi \hspace{-0.15em}\mid_{{\mathcal Q}_s^1}$ being the desired isomorphism between ${\mathcal Q}_s^\e$ and ${\mathcal Q}_{s+1}^\e$ $(\e \in \{0,1\})$.  

  For sake of completeness we point out that  the algebra $\mathcal Q (p)$ can also be filtered by the internal degree of its elements, defined on monomials as follows:  
  \begin{equation}
 \lvert \rho z_I   \rvert =  \begin{cases} \sum_h (2i_h(p-1) + \e_{i_h}), & \text{if $I=(\e_1,i_1;\e_2, i_2; \dots ;\e_m, i_m)$} \\ 0 & \text{if $I=\emptyset$.}
 \end{cases} \end{equation}
 In spite of its geometric importance, the internal degree will not play any role here. 
 
\section{A first descending chain of subalgebras} 

We first need to establish some congruential identities. Let $\N_0$ denote the set of all non-negative integers.
Fixed any prime $p$,  we  write \begin{equation}\label{bag}
 \sum_{i\geq 0} \gamma_i (m) p^i \quad \text{($0 \leq  \gamma_i (m) <p$)}  
\end{equation} 
to denote the $p$-adic expansion of a fixed $m \in \N_0$. The following well-known Lemma is a stardard device to compute  $\bmod \, p$ binomial coefficients.
\begin{lemma}[Lucas' Theorem]\label{l1} For any $(a,b) \in \N_0 \times \N_0$, 
the following congruential identity holds.
\begin{equation}\label{bah}
{a \choose b} \equiv {\prod}_{i \geq 0} {\gamma_i (a)  \choose \gamma_i (b)} \pmod p.
\end{equation}
\end{lemma}
\begin{proof} See \cite[p.~260]{Karaka} or \cite[I 2.6]{StEps}. Equation \ref{bah} follows the usual conventions: ${0 \choose 0}=1$, and  ${l\choose r} =0$ if $0 \leq l <r$.
\end{proof}
Congruence \eqref{bah} immediately yields
\begin{equation}\label{bax}
{p^ra \choose p^rb} \equiv {a \choose b} \pmod p \qquad \text{for every $r \geq 0$},
\end{equation}
since, in both cases, we find on the right side of \eqref{bah} the same products of binomial coefficients, apart from $r$ extra factors all equal to  ${0 \choose 0}=1$.
\begin{corollary}\label{c1}  For any $(\ell,t,h) \in \N_0 \times \N_0 \times \{ 1, \dots , p\}$, the following congruential identity holds.
\begin{equation}\label{plh} \binom{p\ell-h}{pt}\equiv \binom{\ell-1}{ t} \pmod p.
\end{equation}
\end{corollary}
\begin{proof}Since $p\ell-h=(p-h) + p(\ell-1)$, we have $\gamma_0(p\ell-h)=p-h$. Note also that $\gamma_0(pt)=0$. According to Lemma \ref{l1}, we get
\begin{equation}\binom{p\ell-h}{pt} \equiv \binom{p-h}{0} \binom{p(\ell-1)}{pt} \pmod p.
\end{equation}
 We now use Congruence \ref{bax} for $r=1$, and the fact that ${k \choose 0} =1$ for all $k \in \N_0$.
\end{proof}
In order to make notation less cumbersome, we set
\begin{equation}
A(k,j) =  \binom{(p-1)(k-j)-1}{ j}.
\end{equation}
\begin{corollary}\label{c2} Let $(n,j)$ a couple of positive integers. Whenever $j\not \equiv 0 \pmod p$, the binomial coefficient $A(pn,j)$
is divisible by $p$.
\end{corollary}
\begin{proof} If a fixed positive integer $j$ is not divisible by $p$, then there exists a unique couple $(l,h) \in \N \times  \{1, \dots , p-1\}$ such that  $j=pl-h$. 
Hence, setting $$T= (p-1)(n-l)+h-1,$$ we get  \begin{equation}\label{binom} A(pn,j)= \binom{(p-h-1)+pT}{(p-h)+p(l-1)} \equiv  \binom{p-h-1}{p-h} \cdot \binom{T}{l-1} \pmod p \end{equation}
by Lemma \ref{l1} and Equation \eqref{bax}.
Since $p-h-1<p-h$, the first factor on the right side of Equation \eqref{binom} is zero,  so the result follows. 
\end{proof}

\begin{lemma}\label{c3} Let $(s,n,j)$ a triple of positive integers. Whenever $j\not\equiv 0 \pmod{p^s}$,  the binomial coefficient $A(p^sn,j)$ is divisible by $p$. 
\end{lemma}
\begin{proof} We proceed by induction on $s$. The $s=1$ case is essentially Corollary \ref{c2}. 

 Suppose now $s>1$. The hypothesis on $j$ is equivalent to the existence of a suitable  $(b, i) \in \N \times \{ 1, \dots,  p^s-1 \}$ such that $j=p^sb-i$. Likewise, we can write
$i=pl-r$, for a certain $(l,r) \in \{1, \dots, p^{s-2}  \} \times \{0, \dots, p-1 \}$. 

We now distinguish two cases.
If $r=0$, the binomial coefficient $A(p^sn,j)$ has the form $\binom{p\ell-h}{pt}$ where
$$ \ell= (p-1)(p^{s-1}n-p^{s-1}b+l), \qquad h=1, \qquad \text{and} \qquad t= p^{s-1}b-l.$$
By Corollary \ref{c1}, we get
$$ A(p^sn,j) \equiv A (p^{s-1}n, p^{s-1}b-l) \pmod p,$$
and the latter is divisible by $p$ by the inductive hypothesis.

Assume now $1\leq r\leq p-1$. In this case, 
\begin{equation}\label{t'}  A(p^sn,j)= \binom{r-1 +pT'}{r+ p(p^{s-1} b -l)} 
\end{equation}
where $T'=  (p-1)(p^{s-1}n-p^{s-1}b+l)-r$.
Therefore, by Lemma \ref{l1} we get
\begin{equation}\label{ko} A(p^sn,j) \equiv \binom{r-1}{r}\cdot \binom{T'}{p^{s-1}b-l} \pmod p.
\end{equation}
The right side of Equation \ref{ko} vanishes, since $r-1<r$, and the proof is over.
\end{proof}
Lemmas and Corollaries proved so far  will be helpful to reduce, in some particular cases, the number of potentially non-zero binomial coefficients in \eqref{Rkn} and in \eqref{Skn}.
For instance, for any $(h,n) \in \Z \times \N_0$, relations of type $R(\e,p^sh-\alpha_s,p^sn)$, where
$$\alpha_s ={p^s-1 \over p-1} \qquad (s \geq 1),$$ 
only involve generators in the set 
\begin{equation}\label{sem}
\mathcal T_{(\e,s)}= \{ z_{\e,p^sm-\alpha_s} \, \vert \, m \in \Z \}
\end{equation} as stated in the following Proposition.
\begin{proposition}\label{p1} Let $(\e, k,n,s)$ a fixed $4$-tuple in $\{0,1\} \times \Z 	\times \N_0 \times \N$. The polynomial $R(\e, p^sk-\alpha_s, p^sn)$ in \eqref{Rkn} is actually equal to
$$
z_{\e, p^s(pk-1-n)-\alpha_s}z_{0,p^sk-\alpha_s}
\,+\sum_j(-1)^j A(n,j) \, z_{\e,p^s(pk-1-j)-\alpha_s}z_{0,p^s(k-n+j)-\alpha_s}.
$$
\end{proposition}
\begin{proof} By definition (see \eqref{Rkn}), $R(\e, p^sk-\alpha_s,p^sn)$ is equal to
$$
z_{\e,p(p^sk-\alpha_s)-1-p^sn}z_{0,p^sk-\alpha_s}
\,+\sum_l(-1)^l A(p^sn,l) \, z_{\e,p(p^sk-\alpha_s)-1-l}z_{0,p^sk-\alpha_s-p^sn+l}.
$$
According to Lemma \ref{c3},  the only possible non-zero coefficients in  the sum above occur when  $l\equiv 0 \pmod{p^s}$. Thus, we set $l=p^sj$ and write  $R(\e, p^sk-\alpha_s,p^sn)$ as
$$
z_{\e,p(p^sk-\alpha_s)-1-p^sn}z_{0,p^sk-\alpha_s}
\,+\sum_j(-1)^{p^sj} A (p^sn, p^sj)z_{\e,p(p^sk-\alpha_s)-1-p^sj}z_{0,p^sk-\alpha_s-p^sn+p^sj}.
$$
In such polynomial we can replace $z_{\e,p(p^sk-\alpha_s)-1-p^sn}$ and  $z_{\e,p(p^sk-\alpha_s)-1-p^sj}$ by 
$$ z_{\e,p^s(pk-1-n)-\alpha_s} \quad \text{and}  \quad z_{\e,p^s(pk-1-j)-\alpha_s}$$
respectively, since $p\alpha_s+1=p^s+\alpha_s$.
Finally, applying  Equation \eqref{plh} as many times as necessary, and recalling that we are supposing  $p$ odd,  we get
 \begin{equation}\label{apnj} (-1)^{p^sj} A(p^sn,p^sj) \equiv (-1)^j A(n,j) \pmod p.
 \end{equation}
\end{proof}

 As a consequence of Proposition \ref{p1}, the admissible expression of any non-admissible monomial  with label 
$(\e,p^sh_1-\alpha_s;0, p^sh_2-\alpha_s; \dots ;0, p^sh_m-\alpha_s)$ involves only generators in $\mathcal{T}_{(\e,s)}$.

That's the reason why, for any non-negative integer $s$, there is a well-defined  $\F_p$-algebra  ${\mathcal Q}_s^0$ generated by the set   $\{1\} \cup \mathcal{T}_{(0,s)}$ and  subject to
 relations 
 $$R(0,{p^sh-\alpha_s}, p^sn)=0 \quad \forall n\in \N_0.$$

Thus ${\mathcal Q}_0^0$ and  ${\mathcal Q}_1^0$ are the subalgebras of ${\mathcal Q}(p)$ generated by the sets $$\{1\}\cup\{z_{0, h}\, \vert \, h\in\Z\} \quad \text{and} \quad  
\{1\}\cup\{z_{0, ph-1}\, \vert \, h\in\Z\}$$ respectively. The former has been simply denoted by ${\mathcal Q}^0$ in Section~1. The arithmetic identity
\begin{equation}\label{palfa} 
p^{s+1}h-\alpha_{s+1}=p^{s}(ph-1)-\alpha_{s},
\end{equation}
 implies that ${\mathcal Q}_s^0 \supset {\mathcal Q}_{s+1}^0$.

\begin{lemma}\label{lem2} A monomial of type 
\begin{equation}\label{admz} z_I=z_{\e,p^sh_1-\alpha_s}z_{0,p^sh_2-\alpha_s}\cdots \, z_{0,p^sh_m-\alpha_s}
\end{equation} is admissible if and only if $h_i \,\geq \, ph_{i+1}$ for any $i=1, \dots , m-1$. 
\end{lemma}
\begin{proof} Admissibility for a monomial of type \eqref{admz} is tantamount to the condition
$$ p^sh_i - \alpha_s \geq p (p^sh_{i+1} - \alpha_s) \quad \forall i \in \{ 1, \dots, m-1 \}. $$
Inequalities above are equivalent to
$$ h_i  \geq p h_{i+1} - \frac{p^s-1}{p^s} \quad \forall i \in \{ 1, \dots, m-1 \},$$
and the ceiling of the real number on the right side is precisely $ph_{i+1}$.
\end{proof}

\begin{proposition}\label{p2} An $\F_p$-linear basis for ${\mathcal Q}_s^0$ is given by the set $\mathcal B_{{\mathcal Q}_s^0}$ of its monic admissible monomials.
\end{proposition}
\begin{proof} In \cite{CL} it is explained the procedure to express any monomial in $\mathcal Q(p)$ as a sum of admissible monomials. As Proposition \ref{p1} shows, the generalized Adem relations required to complete such procedure starting from a monomial in  ${\mathcal Q}_s^0$   only involve generators actually available in the set at hands.
\end{proof}

So far, we have established the existence of the following descending chain of algebra inclusions: 
$${\mathcal Q}^0 ={\mathcal Q}_0^0 \supset {\mathcal Q}_1^0 \supset {\mathcal Q}_2 ^0\supset \dots \supset {\mathcal Q}_s^0 \supset {\mathcal Q}_{s+1}^0\supset \dots,$$

On the free $\F_p$-algebra $\F_p\langle \{1\} \cup \mathcal{T}_{(0,0)} \rangle$  we now define a monomorphism 
$\Phi$ acting on the generators as follows
\begin{equation}
\Phi (1) =1 \qquad \text{and} \qquad \Phi (z_{0,k})= z_{0,pk-1}.
\end{equation}
We set $\Phi^0= 1_{\F_p\langle \mathcal S_p \rangle}$ and $\Phi^s = \Phi \circ \Phi^{s-1}$ for $s \geq 1$.
\begin{proposition}\label{ppp} For each $s \geq 0$,
\begin{equation}\label{fi1} \Phi^s(z_{0,i_1} \cdots \, z_{0,i_m})= z_{0,p^si_1-\alpha_s} \cdots \; z_{0,p^si_m-\alpha_s},
\end{equation}
and
\begin{equation}\label{fi2}  \Phi^s(R(0,k,n))=R(0,p^sk-\alpha_s, p^sn).
\end{equation}
\end{proposition}
\begin{proof} Equations \eqref{fi1} and \eqref{fi2} are trivially true for $s=0$. For $s\geq 1$ use an inductive argument taking into account \eqref{palfa} and Proposition \ref{p1}.
\end{proof}
\begin{proposition}\label{ppp2} Let $\pi : \F_p \langle \{1\} \cup \mathcal{T}_{(0,0)}\rangle \rightarrow \mathcal Q^0$ be the quotient map.There exists an algebra monomorphism $\phi$ such that the diagram 
\begin{equation}
\begin{diagram}
\node{\F_p\langle \{1\} \cup \mathcal{T}_{(0,0)} \rangle} 
    \arrow{e,t}{\Phi} 
    \arrow{s,l}{\pi} 
\node{\F_p\langle \{1\} \cup \mathcal{T}_{(0,0)} \rangle}
    \arrow{s,r}{\pi}  \\
\node{\mathcal{Q}^0}
    \arrow{e,b,..}{\phi} 
\node{\mathcal{Q}^0}
\end{diagram}
\end{equation}
commutes.
\end{proposition} 
\begin{proof}
By Equation \eqref{fi2}, it follows in particular that 
$$ \Phi (R(0,k,n))=R(0,pk-1, pn).$$
Therefore there exists a well-defined algebra map
$$ \phi:  z_{0,i_1}z_{0,i_2}\cdots z_{0,i_m} \in \mathcal Q^0 \longmapsto z_{0,pi_1-1}z_{0,pi_2-1}\cdots z_{0,pi_m-1} \in \mathcal Q^0.  $$
Such map is injective since the set $\mathcal B_{{\mathcal Q}_s^0}$ -- an $\F_p$-linear basis for $\mathcal Q^0$ according to Proposition \ref{p2} -- is mapped onto admissibles by Lemma \ref{lem2}. 
\end{proof}
\begin{corollary}\label{cz1} The algebra $\mathcal Q^0_s$ is isomorphic to its subalgebra $\mathcal Q^0_{s+1}$.
\end{corollary}
\begin{proof} 
By Propositions \ref{ppp} and \ref{ppp2}, we can argue that $\phi^s (\mathcal Q^0) = \mathcal Q^0_s$. Hence the map
$$ \phi\hspace{-0.15em}\mid_{{\mathcal Q}_s^0} : \Imm \phi^s \lr \Imm \phi^{s+1}$$
gives the desired isomorphism.
\end{proof}
Corollary \ref{cz1} proves Theorem \ref{teo1} for $\e=0$.

\section{A second descending chain of subalgebras}
The aim of this Section is to provide a proof for the $\e=1$ case of Theorem \ref{teo1}.  We choose to follow as close as possible¶ the line of attack put forward in Section 2. 
\begin{proposition}\label{s1} Let $(k,n,s)$ a fixed triple in $\Z 	\times \N_0 \times \N$. In \eqref{Skn} the polynomial $S(1, p^sk, p^sn)$  is actually equal to
$$ z_{1, p^s(pk-n)}z_{1,p^sk} +\sum_j(-1)^{j+1}A(n,j) \, z_{1,p^s(pk-j)}z_{1,p^s(k-n+j)}.$$
\end{proposition}
\begin{proof} By definition (see \ref{Skn}),
\begin{equation}\label{snj} S(1, p^sk,p^sn)=z_{1,p^s(pk-n)}z_{1,p^sk} +\sum_l(-1)^{l+1}A(p^sn,l) z_{1,p^{s+1}k-l}z_{1,p^sk-p^sn+l}.
\end{equation}
According to Lemma \ref{c3},  the only possible non-zero coefficients in  the sum above are those with $l\equiv 0$ mod $p^s$. Setting $l=p^sj$, the polynomial \eqref{snj} becomes
$$ z_{1,p^{s+1}k-p^sn}z_{1,p^sk} +\sum_j(-1)^{p^sj+1}A(p^sn, p^sj) z_{1,p^{s+1}k-p^sj}z_{1,p^sk-p^sn+p^sj}.
$$
 The result now follows from Equation \eqref{apnj}. 
\end{proof}
Proposition \ref{s1} implies that relations of type $S(1,p^sh,p^sn)$ only involve generators of type $z_{1,p^sm}$. therefore the admissible expression of any non-admissible monomial  with label 
$(1,p^sh_1;1, p^sh_2; \dots ;1, p^sh_m)$ only involves generators in the set
\begin{equation}\label{tem}
\mathcal T'_{(1,s)}= \{ z_{1,p^sm} \, \vert \, m \in \Z \}.
\end{equation}
So it makes sense to define ${\mathcal Q}_s^1$ as the $\F_p$-algebra generated by the set   $\{1\} \cup \mathcal T'_{(1,s)}$ and subject to relations $$S(1,{p^sh}, p^sn)=0 \quad \forall n\in \N_0.$$
Each $\mathcal Q^1_s$ is actually a subalgebra of $\mathcal Q(p)$. We have inclusions ${\mathcal Q}_s^1 \supset {\mathcal Q}_{s+1}^1$. In Section~1, the algebra ${\mathcal Q}_0^1$ has been simply denoted  by ${\mathcal Q}^1$.
\begin{lemma}\label{lem31}
A monomial of type  
\begin{equation}\label{z1p} z_{1,p^sh_1}z_{1,p^sh_2}\cdots \, z_{1,p^sh_m}
\end{equation} in ${\mathcal Q}_s^1 \subset \mathcal Q(p)$  is admissible if and only if  
$ h_i  \geq p h_{i+1} +1 \quad \forall i \in \{ 1, \dots, m-1 \}. $
\end{lemma}
\begin{proof}
By definition, the monomial \eqref{z1p} is admissible if and only if $$ p^sh_i  \geq p (p^sh_{i+1}) +1 \quad \forall i \in \{ 1, \dots, m-1 \}. $$
Inequalities above are equivalent to
$$ h_i  \geq p h_{i+1} + \frac{1}{p^s} \quad \forall i \in \{ 1, \dots, m-1 \},$$
and the ceiling of the real number on the right side is precisely $ph_{i+1}+1$.
\end{proof}
\begin{proposition}\label{s2} An $\F_p$-linear basis for ${\mathcal Q}_s^1$ is given by the set $\mathcal B_{{\mathcal Q}_s^1}$ of its monic admissible monomials.
\end{proposition}
\begin{proof} Follows verbatim the proof of Proposition \ref{p2}, just replacing ``Proposition \ref{p1}'' by ``Proposition \ref{s1}'' and $\mathcal Q^0_s$ by  $\mathcal Q^1_s$.
\end{proof}
We are now going to prove that the subalgebras in the descending chain
$$ {\mathcal Q}^1 ={\mathcal Q}_0^1 \supset {\mathcal Q}_1^1 \supset {\mathcal Q}_2 ^1\supset \dots \supset {\mathcal Q}_s^1 \supset {\mathcal Q}_{s+1}^1\supset \dots,
 $$
 are all isomorphic. To this aim we consider  the injective endomorphism  $\Psi$ on the free $\F_p$-algebra 
 $\F_p \langle \{1\} \cup \mathcal T'_{1,0} \rangle$ by setting
 \begin{equation}
\Psi (1) =1 \qquad \text{and} \qquad \Psi (z_{1,k})= z_{1,pk}.
\end{equation}
 
 \begin{proposition}\label{ppp3} Let $\pi' : \F_p \langle \{1\} \cup \mathcal{T'}_{(1,0)}\rangle \rightarrow \mathcal Q^1$ be the quotient map.There exists an algebra monomorphism $\psi$ such that the diagram 
\begin{equation}
\begin{diagram}
\node{\F_p\langle \{1\} \cup \mathcal{T'}_{(1,0)} \rangle} 
    \arrow{e,t}{\Psi} 
    \arrow{s,l}{\pi'} 
\node{\F_p\langle \{1\} \cup \mathcal{T'}_{(1,0)} \rangle}
    \arrow{s,r}{\pi}  \\
\node{\mathcal{Q}^1}
    \arrow{e,b,..}{\psi} 
\node{\mathcal{Q}^1}
\end{diagram}
\end{equation}
commutes.
\end{proposition} 
\begin{proof} Since $ \Psi^s(z_{1,i_1} \cdots \, z_{1,i_m})= z_{1,p^si_1} \cdots \; z_{1,p^si_m}$,
by Proposition \ref{s1} we argue that
\begin{equation}\label{fi4}  \Psi^s(S(1,k,n))=S(1,p^sk, p^sn).
\end{equation}
Therefore there exists a well-defined algebra map
$$ \psi:  z_{1,i_1} \cdots \, z_{1,i_m} \in \mathcal Q^1 \longmapsto z_{1,pi_1} \cdots \, z_{1,pi_m} \in \mathcal Q^1.  $$
Such map is injective since the set $\mathcal B_{{\mathcal Q}_s^1}$ -- an $\F_p$-linear basis for $\mathcal Q^1$ according to Proposition \ref{s2} -- is mapped onto admissibles by Lemma \ref{lem31}. 
\end{proof}
\begin{corollary}\label{cz3} The algebra $\mathcal Q^1_s$ is isomorphic to its subalgebra $\mathcal Q^1_{s+1}$.
\end{corollary}
\begin{proof} 
By Equation \eqref{fi4} and Proposition \ref{ppp3}, we can argue that $\psi^s (\mathcal Q^1) = \mathcal Q^1_s$. Thus, the desired isomorphism is given by
$$ \psi\hspace{-0.15em}\mid_{{\mathcal Q}_s^1} : \Imm \psi^s \lr \Imm \psi^{s+1}.$$
\end{proof}

\section{Further substructures}
For each $s\in \N_0$, we define $V_s$ to be the $\F_p$-vector subspace of  ${\mathcal Q}(p)$ generated by the set  of monomials 
$${\mathcal U}_s=\{z_{1,p^sh_1-\alpha_s}z_{0,p^sh_2-\alpha_s}\cdots \, z_{0,p^sh_m-\alpha_s}\, \vert \; m\geq 2, \; (h_1, \dots ,h_m) \in \Z^m \,  \}.$$

Equation \ref{palfa} implies that  $V_s \supset V_{s+1}$. 
None of the $V_s$'s is a subalgebra of $\mathcal Q(p)$, nevertheless, by Proposition \ref{p1} and the nature of relations \eqref{Rkn} it follows that $V_s$ can be endowed  with  a right ${\mathcal Q}^0_s$-module structure just by considering multiplication in $\mathcal Q(p)$. 
By using once again Lemma \ref{lem2} and the argument along the proof of Proposition \ref{p2}, we get
\begin{proposition}\label{p8} An $\F_p$-linear basis for $V_s$ is given by the set $\mathcal B_{V_s}$ of its monic admissible monomials.
\end{proposition}
\begin{proposition}\label{pult} The map between sets
$$ z_{1,i_1}z_{0,i_2}\cdots \,  z_{0,i_m} \in {\mathcal U}_0 \;  \longmapsto \; z_{1,pi_1-1}z_{0,pi_2-1}\cdots \, z_{0,pi_m-1} \in  {\mathcal U}_0$$
can be extended to a well-defined injective $\F_p$-linear map
$\lambda : V_0 \lr V_0.$
Moreover
\begin{equation}\label{lll}
\lambda^s (V_0) = V_s \subset V_0. 
\end{equation}
\end{proposition}
\begin{proof}
As in the proof of Proposition \ref{ppp}, Equation \ref{palfa}  and Proposition \ref{p1} show that the polynomial $R(\e,k,n) \in \F_p \langle \mathcal S_p \rangle$ in mapped onto $R(\e, p^sk-\alpha_s,p^sn)$ through the $s$-th power of the $\F_p$-linear map
\begin{equation}\label{bda} \Lambda:  z_{\e_11,i_1}z_{\e_2,i_2}\cdots \, z_{\e_m,i_m} \,  \in \,  \F_p \langle \mathcal S_p \rangle \, \longmapsto \, z_{\e_1,pi_1-1}z_{\e_2,pi_2-1}\cdots \, z_{\e_m,pi_m-1}  \, \in \, \F_p \langle \mathcal S_p \rangle.
\end{equation}
Hence there are two maps $\bar{\Lambda}$ and $\lambda$ suct that the diagram \\
\begin{equation}
\begin{tikzcd}
\F_p \langle \mathcal S_p \rangle  \arrow{r}{\Lambda} &  \F_p \langle \mathcal S_p \rangle \\
\F_p \langle \mathcal U_0 \rangle  \arrow{d}{\pi''} \arrow[hook]{u} \arrow[dotted]{r}{\bar{\Lambda}}& F_p \langle \mathcal U_0 \rangle \arrow[hook]{u} \arrow{d}{\pi''}   \\
V_0 \arrow[dotted]{r}{\lambda} & V_0
\end{tikzcd}
\end{equation}
commutes, where $\pi'' : \F_p \langle \mathcal U_0 \rangle \rightarrow V_0$ is the quotient map.
Finally, taking into account Equation \ref{palfa},   one checks that
\begin{equation}\label{lamb}
\lambda^s(z_{1,i_1}z_{0,i_2}\cdots \, z_{0,i_m})= z_{1,p^si_1-\alpha_s}z_{0,p^si_2-\alpha_s}\cdots \, z_{0,p^si_m-\alpha_s}.
\end{equation}
Since Equation \eqref{lamb} implies \eqref{lll}, the proof is over.
\end{proof}
We now introduce a category $\mathcal K$ whose objects are couples $(M,R)$, with $R$ being any ring, and $M$ any right $R$-module. A morphism between two objects $(M,R)$ and $(N,S)$ is given by a couple
$(f, \omega)$ where  $f : M \rightarrow N$ is group homomorphism and $\omega: R \rightarrow S$ is a ring homomorphism, furthermore
$$ f (mr) = f(m) \,  \omega(r)  \qquad \forall (m,r) \in  (M,R).$$
The category $\mathcal K$ is partially ordered by ``inclusions''. More precisely we say that
$$ (M,R) \subseteq (M',R')$$ if $M$ is a subgroup of $M'$ and $R$ is a subring of $R'$.

\begin{theorem}\label{t2} The objects in $\mathcal K$ of the descending chain
$$ ({V}_0, \mathcal Q_0^0) \supset ({V}_1, \mathcal Q_1^0) \supset \dots \supset ({V}_s, \mathcal Q_s^0) \supset ({V}_{s+1}, \mathcal Q_{s+1}^0)  \supset \dots \, $$
are all isomorphic.
\end{theorem}
\begin{proof}
By Proposition \ref{pult} it follows that 
$ \lambda\hspace{-0.15em}\mid_{V_s} : V_s \lr V_{s+1}$
is an isomorphism between $\F_p$-vector spaces. 
Thus, recalling Corollary \ref{cz1}, the desired isomorphism in $\mathcal K$ is given by
$$ (\lambda\hspace{-0.15em}\mid_{V_s}, \phi\hspace{-0.15em}\mid_{{\mathcal Q}_s^0} ) : ({V}_s, \mathcal Q_s^0) \; \lr \;  ({V}_{s+1}, \mathcal Q_{s+1}^0).$$
\end{proof}

\section{A final remark}
Theorem 1.1 in \cite{BolMex} says that no strict algebra monomorphism in $\mathcal Q(p)$ exists when $p$ is odd. Hence there is no chance to find algebra endomorhisms over $\mathcal Q(p)$ extending the maps $\phi$ and $\psi$ defined in Sections 2 and 3 respectively. Just to give an idea about the obstructions you come up with, consider  the $\F_p$-linear map
$$  \Theta : \F_p \langle \mathcal S_p \rangle \; \lr \; \F_p \langle \mathcal S_p \rangle $$
defined on monomials as follows
$$  \Theta (z_{\e_1,i_1} \cdots \, z_{\e_m,i_m} ) =  z_{\e_1,pi_1} \cdots \, z_{\e_m,pi_m}.$$
Neither the map $\Theta$ nor the map $\Lambda$ introduced in Section 4 stabilizes the entire set \eqref{relaz}. Indeed, take for instance 
$$ R(0, 0,0)= z_{0,-1}z_{0,0} \quad \text{and} \quad S(1, 0,0)= z_{1,0}z_{1,0}.$$
The polynomial
\begin{equation}\label{psi0} \Theta (R(0, 0,0)) =z_{0,-p}z_{0,0}
\end{equation}
does not belong to the set $\mathcal R_p$. In fact, the only  polynomial in $\mathcal R_p$ containing \eqref{psi0} as a summand is 
$$R(0,0,p-1)= z_{0,-1-(p-1)}z_{0,0} + z_{0,-1}z_{0,-p+1}. $$
Similarly, 
the polynomial 
$$\label{phi0}\Lambda (S(1, 0,0))= z_{1,-1}z_{1,-1}
$$
does not belong to the set $\mathcal R_p$, since it consists of a single admissible monomial, whereas each element in $\mathcal R_p$ always contains a non-admissible monomial among its summands.


\begin{thebibliography}{Morava}
 
 \bibitem{AK} S. Araki, T. Kudo, \textit{Topology of $H_n$-spaces and $H$\nobreakdash-squaring operations}, Mem. Fac. Sci. Kyusyu Univ. Ser. A {\bf 10} (1956), 85--120.
 
\bibitem{ManuMath}  M. Brunetti, A. Ciampella, L. A. Lomonaco, \textit{The Cohomology of the Universal Steenrod algebra}, Manuscripta Math., {\bf 118} (2005), 271--282. 

\bibitem{AMS}  M. Brunetti, A. Ciampella, L. A. Lomonaco, \textit{An Embedding for the $E_2$-term of the Adams Spectral Sequence at Odd Primes}, Acta Mathematica Sinica, English Series {\bf 22} (2006), no. 6, 1657--1666. 

\bibitem{Colloq}  M. Brunetti, A. Ciampella, \textit{A Priddy-type koszulness criterion for non-locally finite algebras}, Colloquium Mathematicum {\bf 109}  (2007), no. 2, 179--192. 

\bibitem{BLMS}  M. Brunetti, A. Ciampella, L. A. Lomonaco, \textit{Homology and cohomology operations in terms of differential operators}, Bull. London Math. Soc. {\bf 42}  (2010), no. 1, 53--63. 

\bibitem{Viet} M. Brunetti, A. Ciampella, L. A. Lomonaco, \textit{An Example in the Singer Category of Algebras with Coproducts at Odd Primes}, Vietnam J. Math. {\bf 44} (2016), no. 3, 463--476.

\bibitem{BolMex} M. Brunetti, A. Ciampella, L. A. Lomonaco, \textit{Length-preserving monomorphisms for some algebras of operations}, Bol. Soc. Mat. Mex. {\bf 23} (2017), no. 1, 487--500.  

\bibitem{ApplMS}  M. Brunetti, A. Ciampella, \textit{The Fractal Structure of the Universal Steenrod Algebra: an invariant-theoretic description}, Applied Mathematical Sciences, Vol. {\bf 8}  no. 133 (2014), 6681--6687 .

\bibitem{BL} M. Brunetti, L. A. Lomonaco, \textit{Chasing non-diagonal cycles in a certain system of algebras of operations}, Ricerche Mat. {\bf 63} (2014), no. 1, suppl., S57--S68.

\bibitem{RendASFM} A. Ciampella, \textit{On a fractal structure of the universal Steenrod algebra}, Rend. Accad. Sci. Fis.
Mat. Napoli, vol. 81, ({\bf 4}) (2014),   203--207  .


\bibitem{CL} A. Ciampella, L. A. Lomonaco, \textit{The Universal Steenrod 
Algebra at Odd Primes}, Communications in Algebra {\bf 32} (2004), no. 7,  2589--2607. 

\bibitem{Funda} A. Ciampella, L. A. Lomonaco, \textit{ Homological computations in the universal Steenrod algebra}, Fund. Math. {\bf 183} (2004), no. 3, 245--252. 

\bibitem{Karaka} I. Karaca, \textit{Nilpotence relations in the mod p Steenrod algebra}, J. Pure Appl. Algebra {\bf 171} (2002), no. 2--3, 257--264.

\bibitem{Li} A.~Liulevicius, \textit{The factorization of cyclic reduced powers by secondary cohomology operations},  Mem. Amer. Math. Soc. {\bf 42} (1962).

\bibitem{L0} Lomonaco~L.~A.,  
\textit{Dickson invariants and the universal Steenrod algebra.}  
Topology, Proc. 4th Meet., Sorrento/Italy 1988, Suppl. Rend. Circ. Mat. Palermo, II. Ser. {\bf 24} (1990), 429--443.

\bibitem{May} J. P. May, \textit{A General Approach to Steenrod Operations}, Lecture Notes in Mathematics. {\bf 168}, Berlin: Springer, 153--231 (1970). 

\bibitem{May2} J.~P.~May,
\textit{Homology operations on infinite loop spaces}, 
Algebraic topology (Proc. Sympos. Pure Math., Vol. XXII, Univ. Wisconsin, 
Madison, Wis., 1970), pp. 171--185. Amer. Math. Soc., Providence, R.I. 
(1971).

\bibitem{Monks} K. G. Monks, \textit{Nilpotence in the Steenrod algebra}, Bol. Soc. Mat. Mexicana (2) {\bf 37} (1992), no.
1-2, 401--416  (Papers in honor of Jos\'e Adem).

\bibitem{StEps}  N. E. Steenrod, \textit{Cohomology Operations}, lectures written and revised by D. B. A. Epstein, 
Ann. of Math. Studies {\bf 50}, Princeton Univ. Press, Princeton, NJ  (1962). 





\end{thebibliography}
\end{document}